\documentclass[12pt]{amsart}
\usepackage[utf8]{inputenc}

\usepackage{amsxtra,amssymb,amsmath,url,enumitem}

\setenumerate{itemsep=5pt}
\setitemize{itemsep=5pt}

\newlist{enumth}{enumerate}{1}
\setlist[enumth]{label=\emph{(\arabic*)}, ref=(\arabic*)}

\usepackage[utf8]{inputenc}
\usepackage[english]{babel}
\usepackage{fullpage}
\usepackage[colorlinks]{hyperref}
\usepackage{datetime}
\usepackage{mathrsfs}

\shortdate

\DeclareMathSymbol{A}{\mathalpha}{operators}{`A}%
\DeclareMathSymbol{B}{\mathalpha}{operators}{`B}%
\DeclareMathSymbol{C}{\mathalpha}{operators}{`C}%
\DeclareMathSymbol{D}{\mathalpha}{operators}{`D}%
\DeclareMathSymbol{E}{\mathalpha}{operators}{`E}%
\DeclareMathSymbol{F}{\mathalpha}{operators}{`F}%
\DeclareMathSymbol{G}{\mathalpha}{operators}{`G}%
\DeclareMathSymbol{H}{\mathalpha}{operators}{`H}%
\DeclareMathSymbol{I}{\mathalpha}{operators}{`I}%
\DeclareMathSymbol{J}{\mathalpha}{operators}{`J}%
\DeclareMathSymbol{K}{\mathalpha}{operators}{`K}%
\DeclareMathSymbol{L}{\mathalpha}{operators}{`L}%
\DeclareMathSymbol{M}{\mathalpha}{operators}{`M}%
\DeclareMathSymbol{N}{\mathalpha}{operators}{`N}%
\DeclareMathSymbol{O}{\mathalpha}{operators}{`O}%
\DeclareMathSymbol{P}{\mathalpha}{operators}{`P}%
\DeclareMathSymbol{Q}{\mathalpha}{operators}{`Q}%
\DeclareMathSymbol{R}{\mathalpha}{operators}{`R}%
\DeclareMathSymbol{S}{\mathalpha}{operators}{`S}%
\DeclareMathSymbol{T}{\mathalpha}{operators}{`T}%
\DeclareMathSymbol{U}{\mathalpha}{operators}{`U}%
\DeclareMathSymbol{V}{\mathalpha}{operators}{`V}%
\DeclareMathSymbol{W}{\mathalpha}{operators}{`W}%
\DeclareMathSymbol{X}{\mathalpha}{operators}{`X}%
\DeclareMathSymbol{Y}{\mathalpha}{operators}{`Y}%
\DeclareMathSymbol{Z}{\mathalpha}{operators}{`Z}%

\renewcommand{\leq}{\leqslant}
\renewcommand{\geq}{\geqslant}

\numberwithin{equation}{section}

\def\setminus{\mathchoice
    {\mathbin{\vrule height .62ex width 1.61ex depth -.38ex}}% 12
    {\mathbin{\vrule height .62ex width 1.61ex depth -.38ex}}% 12
    {\mathbin{\vrule height .50ex width 0.85ex depth -.28ex}}%  9
    {\mathbin{\vrule height .20ex width 0.570ex depth -.24ex}}%  7
}

\renewcommand{\mathcal}{\mathscr}

\newcommand{\Cc}{\mathbf{C}}

\newcommand{\Zz}{\mathbf{Z}}

\newcommand{\Rr}{\mathbf{R}}

\newcommand{\Ff}{\mathbf{F}}

\newcommand{\proba}{\mathbf{P}}
\newcommand{\expect}{\mathbf{E}}

\newcommand{\charfun}{\mathbf{1}}

\def\loccit{loc.\kern3pt cit.{}\xspace}
\def\cf{see\kern.3em}
\def\Cf{See\kern.3em}
\def\eg{e.g.\kern.3em}

\def\resp{\text{resp.}\kern.3em}

\renewcommand{\rho}{\varrho}

\DeclareMathOperator{\SL}{\mathbf{SL}}

\newcommand{\demi}{{\textstyle{\frac{1}{2}}}}

\DeclareMathSymbol{\gena}{\mathord}{letters}{"3C}
\DeclareMathSymbol{\genb}{\mathord}{letters}{"3E}

\theoremstyle{plain}
\newtheorem{theorem}{Theorem}[section]
\newtheorem*{theorem*}{Theorem}
\newtheorem{lemma}[theorem]{Lemma}

\newtheorem{proposition}[theorem]{Proposition}

\theoremstyle{remark}

\theoremstyle{definition}

\newtheorem{remark}[theorem]{Remark}

\renewcommand{\geq}{\geqslant}
\renewcommand{\leq}{\leqslant}

\newcommand{\ab}{\mathrm{ab}}

\begin{document}

\title{Exponential sums over small subgroups, revisited}

\author{Emmanuel Kowalski}
\address[E. Kowalski]{D-MATH, ETH Z\"urich, R\"amistrasse 101, 8092 Z\"urich, Switzerland} 
\email{kowalski@math.ethz.ch}

\subjclass[2010]{11L07, 11T23}

\keywords{Exponential sums, additive combinatorics, sum-product
  phenomenon, Balog--Szemerédi--Gowers Theorem, multiplicative energy,
  random walks on finite abelian groups}

\begin{abstract}
  This is an expository account of the proof of the theorem of Bourgain,
  Glibichuk and Konyagin which provides non-trivial bounds for
  exponential sums over very small multiplicative subgroups of prime
  finite fields.
\end{abstract}

\maketitle

\begin{flushright}
  \textit{... this peaking of the whale's flukes is perhaps
    the \\
    grandest sight to
    be seen in all animated nature},\\
  H. Melville, \textit{Moby-Dick}, Ch. lxxxvi.
\end{flushright}

\bigskip
\bigskip

\section{Introduction}

In the theory of exponential sums in number theory, the study of
``short'' sums remains one of the most mysterious. Truly robust methods,
suitable for the variety of sums that appear in applications, are
lacking in many cases. 

This note is an exposition of the proof by Bourgain, Glibichuk and
Konyagin of a remarkable estimate of this kind. It concerns exponential
sums over ``small'' subgroups of $\Ff_p^{\times}$, and is especially
noteworthy for the techniques, based on additive combinatorics, which
enter into the proof.

The precise result is the following:

\begin{theorem}[Bourgain, Glibichuk and Konyagin]\label{th-bgk}
  Let $\gamma>0$ be a real number. There exists a real number $\nu>0$,
  depending only on~$\gamma$, such that for any prime number~$p$ and
  any subgroup $H\subset \Ff_p^{\times}$ with $|H|\geq p^{\gamma}$, we
  have
  $$
  \sum_{x\in H}e\Bigl(\frac{ax}{p}\Bigr)\ll |H|p^{-\nu}
  $$
  for any $a\in\Ff_p^{\times}$, where the implied constant depends only
  on~$\gamma$. 
\end{theorem}

Theorem~\ref{th-bgk} has an equivalent formulation in terms of Gauss
sums
$$
G_d(a;p)=\sum_{x\in\Ff_p}e\Bigl(\frac{ax^d}{p}\Bigr)
$$
with exponent $d\mid p-1$. Indeed, considering the subgroup
$$
H_d=\{x^d\,\mid\, x\in\Ff_p^{\times}\}
$$
of order $(p-1)/d$, we have
$$
G_d(a;p)= 1+ \sum_{x\in\Ff_p^{\times}}e\Bigl(\frac{ax^d}{p}\Bigr)
=1+d\sum_{y\in H_d}e\Bigl(\frac{ay}{p}\Bigr)=
1+\frac{p-1}{|H_d|}\sum_{y\in H}e\Bigl(\frac{ay}{p}\Bigr)
$$
since each $y\in H_d$ is of the form $y=x^d$ for $d$ different values
of~$x\in\Ff_p^{\times}$. Hence we see that the estimate of the theorem
is equivalent to the bound $G_d(a;p)\ll p^{1-\nu}$, valid provided
$d\leq (p-1)p^{-\gamma}$ for some $\gamma>0$.

Similarly, let $H$ be a subgroup of~$\Ff_p^{\times}$. We can write
$$
\sum_{y\in H}e\Bigl(\frac{ay}{p}\Bigr) =\frac{|H|}{p-1}\sum_{H\subset
  \ker(\chi)}\sum_{y\in\Ff_p}\chi(y)e\Bigl(\frac{ay}{p}\Bigr),
$$
where $\chi$ runs over the subgroup of characters trivial on~$H$ (which
has order $(p-1)/|H|$); using the fact that Gauss sums for non-trivial
characters have modulus~$\sqrt{p}$, we see that the sums in
Theorem~\ref{th-bgk} have modulus at most $\sqrt{p}$. This is
non-trivial for $|H|$ a bit larger than $\sqrt{p}$. (See
Remark~\ref{rm-comments}, (3) for a different proof of this which does
not use Gauss sums.)

\begin{remark}
  (1) Using similar methods in combination with significant other
  ingredients, a number of generalizations of this bound have been
  obtained, among which we single out the result of
  Bourgain~\cite{bourgain-mordell} where non-trivial estimates are
  obtained for the sums
  $$
  \sum_{x\in \Ff_p^{\times}}e\Bigl(\frac{f(x)}{p}\Bigr)
  $$
  for~$f\in\Zz[X]$ of possibly very large degree, provided the degrees
  of the non-zero monomials appearing in~$f$ satisfy suitable conditions
  relative to~$p$.

  We focus on Theorem~\ref{th-bgk} for definiteness and clarity.

  (2) One can wonder about even smaller subgroups, but some restriction
  is certainly needed since $H$ could be of bounded order. For instance,
  if~$p$ is odd, there is always a subgroup of order~$2$, namely
  $\{-1,1\}$, for which the behavior of the sums is quite clearly rather
  different.

  It would be interesting to see if one could say something interesting
  for subgroups $H$ of size $\asymp (\log p)^{C}$ for some constant
  $C>0$.
  
  % The case of subgroups of order, say, $\log p$, should be accessible
  % to We note that papers of Garcia, Hyde and Lutz~\cite{ghl} and Duke,
  % Garcia and Lutz~\cite{dgl}, among others, have showed that some
  % interesting statistical behavior can be observed in the context of
  % sums over roots of unity of fixed order $d$, for primes
  % $p\equiv 1\mods{d}$ (which ensures that $\Ff_p^{\times}$ contain all
  % the $d$-th roots of unity).
  % %% TODO: mention results on these kinds of questions

  (3) The dependency of the exponent~$\nu$ on~$\gamma$ can be made
  explicit in Theorem~\ref{th-bgk}; currently the sharpest result (whose
  proof involves new ideas) is due to
  Shkredov~\cite[Cor.\,16]{shkredov}. 
\end{remark}

\begin{remark}
  Some of the motivation, generalizations and applications of
  Theorem~\ref{th-bgk} are discussed in a talk at IAS by Bourgain in
  December 2008, which is available online~\cite{bourgain-talk}.
  
  P. Kurlberg~\cite{kurlberg} has already written a detailed account of
  the proof of Theorem~\ref{th-bgk}, from which we benefited a lot. The
  first version of the present text was written as part of lecture notes
  for an introductory course on additive combinatorics taught in the
  Fall Semester 2023 at ETH Zürich (see~\cite{add-comb} for the current
  draft), but the current presentation is also quite different from
  that.

  Some of the changes we make in comparison with the original paper of
  Bourgain, Glibichuk and Konyagin (and with Kurlberg's account) are the
  following:
  \begin{itemize}
  \item The argument, which was originally phrased in terms of
    probability measures on~$\Ff_p$ is presented in probabilistic
    language. At least for some readers (starting from the author), this
    focus brings some additional insights and intuition.
  \item In addition, we order and phrase the main steps of the proof
    rather differently (compare Proposition~\ref{pr-alt-1}
    with~\cite[Prop.\,3.1]{kurlberg}, for instance; these are the places
    in the proof where the sum-product theorem is applied). This is done
    partly to highlight a reading of the proof which has recognizable
    connections with more ``classical'' analytic number theory.
  \item We also include a full proof of one of the two basic ingredients
    from additive combinatorics that occur in the proof of
    Theorem~\ref{th-bgk}. This is a version of the
    Balog--Szemerédi--Gowers Theorem (see Theorem~\ref{th-schoen}
    below), for which Schoen has recently given a short proof
    (see~\cite{schoen}); our presentation is based on an unpublished
    note of B. Green. This proof also has a clear probabilistic flavor,
    and thus fits our presentation very well. (On the other, we only
    quote the sum-product theorem over finite fields of Bourgain, Katz
    and Tao~\cite{bkt}, which is the other key ingredient from additive
    combinatorics.)
  \item On a more technical level, we use the same basic probabilistic
    lemma to verify the assumptions in the two applications of the
    Balog--Szemerédi--Gowers Theorem in the proof (see
    Section~\ref{sec-prob-lemmas}), and we streamline or uniformize a
    few other small steps. This should hopefully make the ideas easier
    to memorize or digest.
  \end{itemize}
\end{remark}

\section*{Notation} We use $f=O(g)$ and $f\ll g$ (or $g\gg f$)
synonymously: for functions $f$ and $g$ defined on a set $X$, this means
that there exists a real number $C\geq 0$, called sometimes the
\emph{implied constant}, such that $|f(x)|\leq g(x)$ for all~$x\in X$.

We denote by $|X|$ the cardinality of a set~$X$.

We denote by $\charfun_Y$ the characteristic function of a subset~$Y$ of
a set~$X$.

We note that although we did not attempt to keep track of the constants
in the final estimate, we have done so for the ``easier'' steps.  The
values of these constants (e.g. in Proposition~\ref{pr-exp-1}) are of
course not very important in themselves.

\section*{Acknowledgements}

We thank B. Green for sending his account of Schoen's result. We also
especially thank all the students of the ``Additive Combinatorics''
class for their interest and active participation in the course, and
C. Bortolotto for organizing the exercise sessions.  Thanks to
A. Gamburd for sending the link to Bourgain's talk~\cite{bourgain-talk}
and to I. Shkredov for pointing out his improved bound
in~\cite{shkredov}.

\section{Preliminaries}

We summarize here the background results used in the proof of
Theorem~\ref{th-bgk}. This section can be skipped until needed during
the proof of the theorem.

\begin{lemma}\label{lm-proba-lower-bound}
  Let~$X$ be a bounded non-negative random variable. Let~$M\geq 0$
  be such that $X\leq M$. Assume that
  $$
  \expect(X)\geq (1-\delta)M 
  $$
  for some $\delta>0$. We then have
  $$
  \proba\Bigl(X\geq (1-\gamma) M\Bigr)\geq 1-\frac{\delta}{\gamma}
  $$
  for any $\gamma$ such that $0<\gamma\leq 1$.

  In particular, if $\expect(X)\geq \alpha^{-1}M$ for some $\alpha\geq
  1$, then
  \begin{equation}\label{eq-proba-lb}
    \proba\Bigl(X\geq \frac{M}{2\alpha}\Bigr)\geq \frac{1}{2\alpha}.
  \end{equation}
\end{lemma}

\begin{proof}
  We use Chebychev's inequality to obtain the complementary upper-bound:
  \begin{align*}
    \proba\Bigl(X\leq (1-\gamma) M\Bigr)=\proba(M-X\geq \gamma M)
    \leq \frac{\expect(M-X)}{\gamma M}\leq \frac{\delta}{\gamma}.
  \end{align*}

  In the final assertion, we have $1-\delta=\alpha^{-1}$ and
  $1-\gamma=1-\demi\alpha^{-1}$, so that
  $$
  1-\frac{\delta}{\gamma}=\frac{\demi\alpha^{-1}}{1-\demi\alpha^{-1}}\geq
  \frac{1}{2\alpha},
  $$
  and the second inequality follows.
\end{proof}

%   $\alpha\geq 1$ such that
%   $$
%   \expect(X)\geq \frac{M}{\alpha}.
%   $$

%   We have
%   \begin{equation}\label{eq-proba-lb}
%     \proba\Bigl(X\geq \frac{M}{2\alpha}\Bigr)\geq \frac{1}{2\alpha}.
%   \end{equation}
% \end{lemma}

% \begin{proof}
%   Let~$\varphi$ be the characteristic function of the event
%   $\{X\geq \demi \alpha^{-1}M\}$.  Using the bounds
%   $$
%   0\leq (1-\varphi) X\leq \frac{M}{2\alpha}, \quad\quad 0\leq \varphi
%   X\leq M\varphi,
%   $$
%   and the fact that $\expect(\varphi)=\proba(M\geq \demi\alpha^{-1}M)$,
%   we obtain the inequalities
%   $$
%   \frac{M}{\alpha}\leq \expect(X)=\expect((1-\varphi)
%   X)+\expect(\varphi X)\leq \frac{M}{2\alpha}+M\proba(M\geq
%   \demi\alpha^{-1}M),
%   $$
%   from which the conclusion follows.
% \end{proof}

We now discuss the version of the Balog--Szemerédi--Gowers Theorem that
we will use. We first fix some notation, to be used throughout.

Given a group~$G$ (not necessarily abelian, although this will be the
case in the applications below) and finite subsets $A$ and $B\subset G$,
we denote by $r_{A\cdot B}$ the \emph{representation function} for the
product set $A\cdot B=\{ab\,\mid\, (a,b)\in A\times B\}$, namely
$$
r_{A\cdot B}(x)=\sum_{\substack{(a,b)\in A\times B\\ab=x}}1.
$$

This function satisfies $0\leq r_{A\cdot B}(x)\leq |A|$ for all~$x\in
G$, and
$$
\sum_{x\in G}r_{A\cdot B}(x)=|A||B|.
$$

Moreover, its second moment is the so-called \emph{multiplicative
  energy} (or just \emph{energy}) of $(A,B)$, which we denote~$E(A,B)$:
$$
E(A,B)=\sum_{x\in G}r_{A\cdot B}(x)^2=|\{(a_1,a_2,b_1,b_2)\in A^2\times
B^2\,\mid\, a_1b_1=a_2b_2\}|.
$$

If~$A$ and $B$ are non-empty, we denote by $e(A,B)$ the normalized
energy, defined by
$$
e(A,B)=\frac{E(A,B)}{(|A||B|)^{3/2}}.
$$

Finally, we denote by $A^{-1}$ the set of inverses of elements
of~$A$. If~$G$ is abelian, then since $ab=cd$ is equivalent
to~$ac^{-1}=db^{-1}$, it follows that $E(A,A)=E(A,A^{-1})$.

\begin{theorem}\label{th-schoen}
  Let~$G$ be a group and~$A\subset G$ a non-empty finite
  subset. Let~$\alpha\geq 1$ be such that $e(A)\geq \alpha^{-1}$. There
  exists a subset~$B\subset A$ such that
  \begin{equation}\label{eq-bgs-1}
    |B|\geq \frac{|A|}{4\alpha},\quad\quad
    |B\cdot B^{-1}|\leq
    %% \alpha^5|A|\leq
    2^{14}\alpha^6|B|,
  \end{equation}
  where the implied constant is absolute.
\end{theorem}

We will give the proof below.
%%%

The last (and crucial) part of the proof is the sum-product theorem of
Bourgain, Katz and Tao~\cite{bkt}. 

\begin{theorem}[Bourgain--Katz--Tao]\label{th-sum-product}
  For any $\gamma>0$, there exists $\delta>0$ such that for any prime
  number~$p$ and any set $A\subset \Ff_p$ such that
  $|A|\leq p^{1-\gamma}$, we have
  \begin{equation}\label{eq-sum-product-fp}
    \max(|A+A|,|A\cdot A|)\gg |A|^{1+\delta},
  \end{equation}
  where the implied constant depends only on~$\gamma$.
\end{theorem}

\begin{remark}
  The original version of the theorem includes also the assumption that
  $|A|\geq p^{\gamma}$, but this was found to be unnecessary by Konyagin
  (although it would pose no problem in the application to
  Theorem~\ref{th-bgk}). Two proofs, written in similar style to this
  paper, can be found in the lecture notes~\cite[\S\,4.2]{add-comb}
  (besides the proof in~\cite{bkt}, these notes contain a proof based on
  ideas of Breuillard~\cite{breuillard} related to growth in the
  affine-linear group).
\end{remark}

We finish this section by giving the proof of Theorem~\ref{th-schoen},
following essentially a write-up by B. Green of the argument of
Schoen~\cite{schoen}.  Again, readers who want to focus on the proof of
Theorem~\ref{th-bgk} may skip to the beginning of the next section.

The key step is to find a large subset~$X$ of~$A$ such that the elements
of~$X\cdot X^{-1}$ have a large number of representations as elements of
$A\cdot A^{-1}$. The precise statement is the following:

\begin{proposition}\label{pr-interm}
  Let~$G$ be a group and~$A\subset G$ a non-empty finite
  subset. Let~$\alpha\geq 1$ be such that $e(A)\geq \alpha^{-1}$. Fix a
  real number $\delta$ such that $0<\delta<1$.  Denote by~$r$ the
  representation function for $A\cdot A^{-1}$.

  There exists~$x\in G$ such that
  \begin{equation}\label{eq-schoen-1}
    |A\cap A\cdot x|\geq\frac{|A|}{2\alpha}
  \end{equation}
  and
  \begin{equation}\label{eq-schoen-2}
    \Bigl|\Bigl\{(a,b)\in (A\cap A\cdot x)^2\,\mid\, r(ab^{-1})\geq
    \frac{\delta|A|}{2\alpha^2} \Bigr\}
    \Bigr|\geq (1-\delta)|A\cap A\cdot x|^2.
  \end{equation}
\end{proposition}

\begin{proof}
  The key idea is to take~$x$ ``at random'', but not according to the
  uniform probability measure on~$G$. Rather, we pick a given element
  $x$ with probability proportional to~$r(x)$. More precisely, since
  $$
  \sum_{x\in G}r(x)=|A||A^{-1}|=|A|^2,
  $$
  we let~$X$ be a $G$-valued random variable such that 
  $$
  \proba(X=x)=\frac{r(x)}{|A|^2}
  $$
  for any~$x\in G$.  We further denote $B=A\cap A\cdot X$, which is a
  random subset of~$G$, contained in~$A$.

  Let~$\gamma>0$ be a parameter to be chosen later. We define
  $$
  Y=\{(a,b)\in A\times A\,\mid\, r(ab^{-1})<\gamma|A|\}.
  $$
  
  We will show that for~$\gamma=\delta/(2\alpha^2)$, the inequality
  \begin{equation}\label{eq-schoen-3}
    \expect\Bigl( |B|^2 -\delta^{-1}|(B\times B)\cap Y| \Bigr)\geq
    \frac{|A|^2}{2\alpha^2}
  \end{equation}
  holds. It implies the existence of some element $x\in G$ such that
  $$
  |A\cap A\cdot x|^2 -\delta^{-1}|(A\cap A\cdot x)^2\cap Y|\geq
  \frac{|A|^2}{2\alpha^2},
  $$
  and from this we deduce, on the one hand, that $|A\cap A\cdot x|^2\geq
  |A|^2/(2\alpha^2)$, which implies~(\ref{eq-schoen-1}), and on the
  other hand that
  $$
  |(A\cap A\cdot x)^2\cap Y|\leq \delta|A\cap A\cdot x|^2,
  $$
  which is equivalent to~(\ref{eq-schoen-2}).

  To prove~(\ref{eq-schoen-3}), we first find a lower-bound
  for~$\expect(|B|^2)$. By the Cauchy--Schwarz inequality, we have
  $\expect(|B|^2)\geq \expect(|B|)^2$, and the expectation of the size
  of~$B$ is
  $$
  \expect(|B|)=\sum_{a\in A}\proba(a\in A\cdot X)= \sum_{a\in
    A}\sum_{b\in A}\proba(X=b^{-1}a)= \frac{1}{|A|^2}\sum_{a\in
    A}\sum_{b\in A}r(b^{-1}a),
  $$
  by definition of~$X$. By replacing $r(b^{-1}a)$ by its definition, we
  compute
  $$
  \frac{1}{|A|^2}\sum_{a\in A}\sum_{b\in A}r(b^{-1}a)=
  \frac{1}{|A|^2}\sum_{a\in A}\sum_{b\in A}\sum_{\substack{(x,y)\in
      A^2\\xy^{-1}=b^{-1}a}}1=\frac{E(A,A)}{|A|^2}=|A|e(A).
  $$

  Using the assumption $e(A)\geq \alpha^{-1}$, we therefore get the
  lower bound
  $$
  \expect(|B|^2)\geq \frac{|A|^2}{\alpha^2}.
  $$

  We now handle separately an upper bound for the expectation of
  $(B\times B)\cap Y$. We simply write
  $$
  \expect(|(B\times B)\cap Y|)\leq |A|^2\max_{(a,b)\in Y}
  \proba(\{a,b\}\subset B),
  $$
  and estimate the probability that $\{a,b\}\subset B$ for each
  $(a,b)\in Y$ separately. Since $Y\subset A^2$, this is
  $$
  \proba(a\in B\text{ and } b\in B)=\proba(a\in A\cdot X\text{ and }
  b\in A\cdot X) =\proba(X\in A^{-1}\cdot a\cap A^{-1}\cdot b).
  $$

  From the crude bound $r(x)\leq |A|$, it follows that
  $\proba(X=x)\leq 1/|A|$ for any~$x\in G$, and we deduce that
  $$
  \proba(X\in A^{-1}\cdot a\cap A^{-1}\cdot b)\leq \frac{1}{|A|}
  |A^{-1}\cdot a\cap A^{-1}\cdot b|.
  $$
  
  We now note that
  $$
  |A^{-1}\cdot a\cap A^{-1}\cdot b|=|\{(x,y)\in A^2\,\mid\,
  xy^{-1}=ab^{-1}\}|
  $$
  (because of the bijection $f$ which sends an element
  $w\in A^{-1}\cdot a\cap A^{-1}\cdot b$ to $(aw^{-1},bw^{-1})$, with
  inverse $(x,y)\mapsto a^{-1}x=b^{-1}y$).  Thus we get
  $$
  \proba(a\in B\text{ and } b\in B) \leq
  \frac{1}{|A|}\sum_{\substack{(x,y)\in
      A^2\\xy^{-1}=ab^{-1}}}1=\frac{r(ab^{-1})}{|A|},
  $$
  and by definition of $Y$, this is $<\gamma |A|$.
  % $$
  % \expect(|(B\times B)\cap Y|)\leq |A|^2\max_{(a,b)\in Y}
  % \proba(\{a,b\}\subset B)\leq \gamma|A|^2
  % $$
  % by definition of~$Y$. 
  Thus we have
  $$
  \expect\Bigl( |B|^2 -\delta^{-1}|(B\times B)\cap Y| \Bigr)\geq
  \frac{|A|^2}{\alpha^2}-\frac{\gamma |A|^2}{\delta},
  $$
  and this is $\geq |A|^2/(2\alpha^2)$ if we take
  $\gamma=\delta/(2\alpha^2)$, as claimed.
\end{proof}

\begin{proof}[Proof of Theorem~\ref{th-schoen}]
  We apply Proposition~\ref{pr-interm} with~$\delta=1/10$; we denote by
  $C$ the set $A\cap A\cdot x$ which it provides, and let
  $$
  Y=\Bigl\{y\in G\,\mid\, r(y)\geq \frac{\delta|A|}{2\alpha^2}\Bigr\},
  $$
  where~$r$ is again the representation function for~$A\cdot A^{-1}$. We
  note that
  \begin{equation}\label{eq-schoen-4}
    |Y|\leq 20\alpha^2|A|
  \end{equation}
  by Chebychev's inequality.  Further, for any element $a\in A$, we
  denote by $N(a)$ the set of $b\in C$ such that $ab^{-1}\in Y$.

  We have $0\leq |N(c)|\leq |C|$ for any~$c\in C$; moreover,
  by~(\ref{eq-schoen-2}), we have
  $$
  \frac{1}{|C|} \sum_{c\in C}|N(c)|\geq (1-\delta)|C|,
  $$
  and this implies that $N(c)$ must often be quite close to its maximal
  value. Precisely, from Lemma~\ref{lm-proba-lower-bound} (with $X$ the
  random variable $c\mapsto N(c)$ on $C$ with uniform probability), we
  get
  $$
  |\{c\in C\,\mid\, N(c)\geq (1-\gamma )|C|\}|\geq
  \Bigl(1-\frac{\delta}{\gamma}\Bigr) |C|,
  $$
  whenever $0<\gamma<1$.  
  % for any $\gamma>0$, we have
  % \begin{align*}
  %   \frac{1}{|C|}|\{c\in C\,\mid\, |N(c)|< (1-\gamma) |C|\}| &=
  %   \frac{1}{|C|}\sum_{|C|-|N(c)|>\gamma |C|}1
  %   \\
  %   &\leq \frac{1}{\gamma|C|}
  %   \times \frac{1}{|C|}
  %   \sum_{c\in C} (|C|-|N(c)|)\leq \frac{\delta}{\gamma},
  % \end{align*}
  % and
  Taking $\gamma=\sqrt{\delta}$, we find that there are at least
  $(1-\sqrt{\delta})|C|$ elements of~$C$ such that
  $|N(c)|\geq (1-\sqrt{\delta})|C|$.

  Let~$B$ be the subset of~$C$ (hence of~$A$) defined by this condition
  on~$N(c)$; since Proposition~\ref{pr-interm} implies
  that~$|C|\geq |A|/(2\alpha)$, we already get
  $$
  |B|\geq (1-\sqrt{\delta})|C|\geq \frac{|C|}{2}\geq
  \frac{|A|}{4\alpha}.
  $$

  To conclude the proof, we claim that
  \begin{equation}\label{eq-schoen-5}
    B\cdot B^{-1}\subset \Bigl\{x\in G\,\mid\, s(x)\geq \frac{|C|}{3}\Bigr\},
  \end{equation}
  where~$s$ is the representation function for $Y\cdot Y^{-1}$. Assuming
  this, we observe that the right-hand set satisfies
  $$
  \Bigl|\Bigl\{
  x\in G\,\mid\, s(x)\geq \frac{|C|}{3}
  \Bigr\}\Bigr|
  \leq
  \frac{3|Y|^2}{|C|}
  $$ 
  (by Chebychev's inequality again).
  Using~$|C|\geq |A|/(2\alpha)$ together with~(\ref{eq-schoen-4}), we
  deduce
  $$
  |B\cdot B^{-1}|\leq \frac{3|Y|^2}{|C|}\leq 6\cdot 20^2\cdot
  \alpha^5|A|\leq 4\cdot 6\cdot 20^2\cdot \alpha^6|B|\leq 2^{14}|B|,
  $$
  which finishes the proof of the theorem.

  To prove~(\ref{eq-schoen-5}), pick any $a$ and~$b$ in~$B$; we need a
  lower bound for $s(ab^{-1})$, or in other words for the size of the
  set
  $$
  \{(u,v)\in Y\times Y\,\mid\, uv^{-1}=ab^{-1}\}.
  $$

  There is an injective map
  $$
  N(a)\cap N(b)\to \{(u,v)\in Y\times Y\,\mid\, uv^{-1}=ab^{-1}\}
  $$
  defined by $f(z)=(az^{-1},bz^{-1})$ (the crucial point here is that
  this map is well-defined: we have $(az^{-1},bz^{-1})\in Y\times Y$ by
  definition of~$N(a)$ and $N(b)$). Hence
  $s(ab^{-1})\geq |N(a)\cap N(b)|$. But, by definition, $|N(a)|$
  and~$|N(b)|$ are very large, and so is their intersection. In fact, we
  get
  $$
  |N(a)\cap N(b)|\geq (1-2\sqrt{\delta})|C|\geq \frac{|C|}{3},
  $$
  (recall that $\delta=1/10$), so that $s(ab^{-1})\geq |C|/3$, as
  desired.
\end{proof}

\section{Two probabilistic constructions}
\label{sec-constructs}

We already mentioned that we will present the proof of
Theorem~\ref{th-bgk} in probabilistic language. This relies on two
elementary constructions which we present here, in greater generality
than required.

We consider a finite group~$G$. Given a $G$-valued random variable~$X$
(defined on some probability space~$\Omega$ which we need not specify
precisely), we will denote by~$\rho_X$ its ``density'' function, i.e.,
$\rho_X\colon G\to\Rr$ is the function such that $\rho_X(x)=\proba(X=x)$
for all~$x\in X$.

\textbf{Stepping.} We say that a $G$-valued random variable~$Y$ is a
\emph{stepping} of~$X$ if~$Y=X_1X_2^{-1}$, where~$(X_1,X_2)$ are
independent random variables, both independent of~$X$ and distributed
like~$X$. In particular, $X$ and~$Y$ are then independent. We have
$$
\rho_Y(y)=\proba(Y=y)=\proba(X_1X_2^{-1}=y)=\sum_{x\in
  G}\proba(X=x)\proba(X=x^{-1}y),
$$
and in particular
\begin{equation}\label{eq-rho0}
  \rho_Y(0)=\sum_{x\in G}\proba(X=x)^2.
\end{equation}

Applying the Cauchy--Schwarz inequality to the formula for $\rho_Y(x)$,
we see that $\rho_Y(x)\leq \rho_Y(0)$ for all~$x\in G$.

\begin{remark}
  In additive notation, we have $Y=X_1-X_2$ with $(X,X_1,X_2)$
  independent and identically distributed.
\end{remark}

\par
\medskip
\par

% 1851
% This peaking of the whale's flukes is perhaps the grandest sight
% to be seen in all animated nature.
% H. Melville, Moby-Dick, Ch. lxxxvi. 420
\textbf{Peaking.} We now assume that~$G$ is commutative, with additive
notation, and we denote by $\widehat{G}$ its character group. For any
$G$-valued random variable~$X$, we denote by $\varphi_X$ the
``characteristic function'' of~$X$ (in the probabilistic sense, hence
essentially its Fourier transform), namely the function on $\widehat{G}$
defined by
$$
\varphi_X(\chi)=\expect(\chi(X))
$$
for $\chi\in\widehat{G}$. We have $\varphi_{-X}=\overline{\varphi_X}$,
and if $X_1$ and $X_2$ are independent, then
$\varphi_{X_1+X_2}=\varphi_{X_1}\varphi_{X_2}$.

Let now~$Y=X_1-X_2$ be a stepping of~$X$. According to the above, we
have~$\varphi_Y=|\varphi_X|^2$.  In particular, since
$\varphi_Y=|\varphi_X|^2\geq 0$, and since $\varphi_Y(0)=1$, we can
consider a random variable $\widehat{Y}$ on~$\widehat{G}$ such that
$$
\proba(\widehat{Y}=\chi)=\frac{\varphi_Y(\chi)}{M_X}=
\frac{|\varphi_X(\chi)|^2}{M_X}
$$
for $\chi\in\widehat{G}$, where
$$
M_X=\sum_{\chi\in \widehat{G}}|\varphi_X(\chi)|^2.
$$

Moreover, we may (and do) insist that~$\widehat{Y}$ is independent from
$(X,X_1,X_2)$, hence also from~$Y$.  (Similarly, whenever we consider
$\widehat{Z}$ for some other random variable~$Z$, it will be understood
that $\widehat{Z}$ is independent of any previously described random
variables.)

Intuitively, the random variable $\widehat{Y}$ emphasizes the characters
$\chi$ where $\varphi_X(\chi)$ is large, and for this reason we will say
that $\widehat{Y}$ is a \emph{peaking} of~$Y$, or of~$X$.

\begin{remark}
  If $G=\Zz/q\Zz$ for some integer~$q\geq 1$, we can identify as usual
  the character group with~$G$ by associating to $a\in\Zz/q\Zz$ the
  character $x\mapsto e(ax/q)$. Thus we also identify the characteristic
  function $\varphi_X$ with a function $\Zz/q\Zz\to \Cc$, with
  $$
  \varphi_X(a)=\expect\Bigl(e\Bigl(\frac{aX}{q}\Bigr)\Bigr).
  $$
\end{remark}

Steppings and peakings are related by a simple but crucial formula,
which reflects the Fourier duality. We identify as usual the dual group
of $\widehat{G}$ with $G$, the element $x\in G$ corresponding to the
character $\chi\mapsto \chi(x)$ of $\widehat{G}$. 

\begin{lemma}\label{lm-bgk-link}
  Let~$G$ be a finite commutative group. For any $G$-valued random
  variable~$X$, with stepping~$Y$ and peaking~$\widehat{Y}$, and for any
  $y\in G$, we have
  $$
  \rho_Y(y)=\frac{M_X}{|G|}\varphi_{\widehat{Y}}(y),
  $$
  where the characteristic function of~$\widehat{Y}$ is identified with
  a function on~$G$.
  % Moreover
  % $$
  % \expect(\rho_Y(XY))=\rho_Y(0) \expect(|\varphi_X(X\widehat{Y})|^2).
  % $$
\end{lemma}

\begin{proof}
  We use the orthogonality of characters to represent the
  (set-theoretic!) characteristic function of an element~$y\in G$ by
  $$
  \frac{1}{|G|}\sum_{\chi\in\widehat{G}}\chi(x-y) =\begin{cases}
    1&\text{ if } x=y\\
    0&\text{ if } x\not=y,
  \end{cases}
  $$
  and get
  $$
  \rho_Y(y)=\expect\Bigl(
  \frac{1}{|G|}\sum_{\chi\in\widehat{G}}\chi(Y-y)\Bigr)
  =\frac{1}{|G|}\sum_{\chi\in\widehat{G}} \chi(-y)\varphi_Y(\chi)=
  \frac{M_X}{|G|}\varphi_{\widehat{Y}}(-y),
  $$
  by definition of $\widehat{Y}$.  This proves the lemma since
  $\rho_Y(-y)=\rho_Y(y)$.
\end{proof}

In particular, we note the formula
\begin{equation}\label{eq-rho1}
  \rho_Y(0)=\frac{M_X}{|G|}.
\end{equation}

\begin{remark}
  If $X$ is uniformly distributed on~$G$, then $Y$ is also uniformly
  distributed on~$G$, and $\widehat{Y}$ is a Dirac mass at the unit
  element~$1$ of~$G$. Conversely, if $X$ is a Dirac mass at
  some~$x\in G$, then $Y$ is a Dirac mass at~$1$, and $\widehat{Y}$ is
  uniformly distributed on~$\widehat{G}$.
\end{remark}

\section{Probabilistic lemmas}\label{sec-prob-lemmas}

In order to apply Theorem~\ref{th-schoen}, we will use two lemmas giving
probabilistic conditions that guarantee large energy. We use the
definition of a ``stepping'' of a random variable from the previous
section.

\begin{lemma}\label{lm-random-energy}
  Let~$G$ be a finite group and let~$A$ be a non-empty subset
  of~$G$. Let~$X$ be a $G$-valued random variable and~$Y$ a stepping
  of~$X$. We assume that $\beta\geq 1$ is such that
  $$
  \expect(r_{A\cdot A^{-1}}(X))\geq  \beta^{-1} |A|.
  $$
  
  We then have
  $$
  e(A)\geq \frac{1}{4\beta^{4}\rho_Y(0)|A|}.
  $$
\end{lemma}

\begin{proof}
  Let
  $$
  L=\{x\in G\,\mid\, r_{A\cdot A^{-1}}(x)\geq \demi \beta^{-1}|A|\},
  $$
  so that we have the lower-bound
  $$
  E(A)=\sum_{x\in G}r_{A\cdot A^{-1}}(x)^2\geq \sum_{x\in L}r_{A\cdot
    A^{-1}}(x)^2\geq \beta^{-2}|A|^2|L|.
  $$

  Noting that~$r_{A\cdot A^{-1}}(x)\leq |A|$ for all~$x$, the assumption
  implies that
  $$
  \proba(L)=\proba\Bigl(r_{A\cdot A^{-1}}(X)\geq
  \frac{|A|}{2\beta}\Bigr) \geq \frac{1}{2\beta}
  $$
  (see~(\ref{eq-proba-lb})), but the Cauchy--Schwarz inequality and
  positivity imply that
  $$
  \proba(L)=\sum_{x\in L}\proba(X=x)\leq |L|^{1/2}\Bigl(\sum_{x\in
    G}\proba(X=x)^2\Bigr)^{1/2}=|L|^{1/2}\rho_Y(0)^{1/2},
  $$
  and hence $|L|\geq (2\beta)^{-2}\rho_Y(0)^{-1}$. The previous lower-bound
  gives
  $$
  E(A)\geq 2^{-2}\beta^{-4}\rho_Y(0)^{-1}|A|^2,
  $$
  which implies the desired result.
\end{proof}

The second and final lemma uses this to conclude that the energy of the set
of ``elements with large probability'' will be big if those sets are of
``typical'' size.

\begin{lemma}\label{lm-stepping}
  Let~$G$ be a finite group. Let~$X$ be a $G$-valued random variable and
  let~$Y=X_1X_2^{-1}$ be a stepping of~$X$.  Let~$\alpha\geq 1$ and
  define
  $$
  A=\Bigl\{x\in G\,\mid\, \proba(Y=x)\geq
  \frac{\rho_Y(0)}{\alpha}\Bigr\}.
  $$

  Let~$B\subset A$ and let~$\beta>0$ be such that
  $$
  |B|\geq \frac{1}{\beta\rho_Y(0)}.
  $$

  We have then
  $$
  e(B)\geq \frac{1}{4\alpha^9\beta^4}.
  $$
\end{lemma}

\begin{proof}
  Let~$r=r_{B\cdot B^{-1}}$ be the representation function
  for~$B\cdot B^{-1}$.  We have
  $$
  \expect(r(Y))=\sum_{a,b\in B}\proba(Y=ab^{-1})=
  \sum_{a,b\in A}\proba(X_1a^{-1}=X_2b^{-1}),
  $$
  and this implies that
  \begin{align*}
    \expect(r(Y))&=\sum_{y\in G}\sum_{a,b\in B}\proba(X_1a^{-1}=y\text{
      and } X_2b^{-1}=y)
    \\
    &=\sum_{y\in G}\sum_{a,b\in B}\proba(X_1a^{-1}=y)\proba(X_2b^{-1}=y)
    =\sum_{y\in G}\proba(X_1\in yB)^2.
  \end{align*}

  The ``reversed'' Cauchy--Schwarz inequality now shows that for any
  choice of $f(y)\geq 0$ for $y\in G$, not all zero, we have
  $$
  \expect(r(Y))\geq \frac{V^2}{W}
  $$
  with
  $$
  V=\sum_{y\in G}f(y)\proba(X_1\in yB), \quad\quad
  W=\sum_{y\in G}f(y)^2.
  $$
  
  We pick $f(y)=\proba(X_2=y)$; in this case, we have
  $$
  V=\proba(Y\in B),\quad\quad W=\proba(Y=0),
  $$
  and therefore
  $$
  \expect(r(Y))\geq \frac{\proba(Y\in B)^2}{\rho_Y(0)}\geq
  \frac{\rho_Y(0)}{\alpha^2} |B|^2,
  $$
  where the last step follows from the assumption that $B\subset A$, so
  that $\proba(Y=y)\geq \alpha^{-1}\rho_Y(0)$ for $y\in B$.  Since we
  also assumed that $\rho_Y(0)|B|\geq \beta^{-1}$, this gives
  $\expect(r(Y))\geq \alpha^{-2}\beta^{-1}|B|$.

  Applying Lemma~\ref{lm-random-energy} to the random variable $Y$ and
  the set~$B$, we get
  $$
  e(B)\geq \frac{1}{4\alpha^8\beta^4\rho_Z(0)|B|},
  $$
  where~$Z$ is a stepping of~$Y$.  But we have
  $$
  \rho_Z(0)=\proba(Z=0)=\sum_{y\in G} \proba(Y=y)^2\leq
  \proba(Y=0)\sum_{y\in G}\proba(Y=y)=\proba(Y=0)=\rho_Y(0),
  $$
  and thus $\rho_Z(0)|B|\leq\rho_Y(0)|A|$, which is $\leq \alpha$ by
  Chebychev's inequality, so we get finally the lower bound
  $$
  e(B)\geq \frac{1}{4\alpha^8\beta^4\rho_Y(0)|B|}\geq
  \frac{1}{4\alpha^9\beta^4},
  $$
  as claimed.
\end{proof}

\section{Main steps of the proof}

We will describe in this section the strategy of the proof of
Theorem~\ref{th-bgk}, extracting two intermediate steps before the final
conclusion.

\par
\medskip
\par
\textbf{Step 1.} The first step is an estimate for a specific average of
values of the discrete Fourier transform of random variables on~$\Ff_p$,
which involves the ``peaking'' of Section~\ref{sec-constructs}.

% To state it, we introduce some notation.  For any
% $\Ff_p$-valued random variable~$X$, we denote by $\varphi_X$ the
% ``characteristic function'' of~$X$ (in the probabilistic sense, hence
% essentially its Fourier transform), namely the function on $\Ff_p$
% defined by
% $$
% \varphi_X(a)=\expect\Bigl(e\Bigl(\frac{ax}{p}\Bigr)\Bigr)
% $$
% for $a\in\Ff_p$. We have $\varphi_{-X}=\overline{\varphi_X}$, and if
% $X_1$ and $X_2$ are independent, then
% $\varphi_{X_1+X_2}=\varphi_{X_1}\varphi_{X_2}$; in particular, for a
% stepping~$Y$ of~$X$, we have~$\varphi_Y=|\varphi_X|^2$. 

% Let~$X$ be an~$\Ff_p$-valued random variable~$X$ and~$Y=X_1-X_2$ a
% stepping of~$X$. Since the characteristic function of~$Y$ is
% $|\varphi_X|^2\geq 0$, and since $\varphi_Y(0)=1$, we can consider a
% random variable $\widehat{Y}$ on~$\Ff_p$ such that
% $$
% \proba(\widehat{Y}=a)=\frac{\varphi_Y(a)}{M_X}=\frac{|\varphi_X(a)|^2}{M_X}
% $$
% for $a\in\Ff_p$, where
% $$
% M_X=\sum_{a\in \Ff_p}|\varphi_X(a)|^2.
% $$

% Moreover, we may (and do) insist that~$\widehat{Y}$ is independent from
% $(X,X_1,X_2)$, hence also from~$Y$.  (Similarly, whenever we consider
% $\widehat{Z}$ for some other random variable~$Z$, it will be understood
% that $\widehat{Z}$ is independent of any previously described random
% variables.)

\begin{proposition}\label{pr-alt-1}
  Let $p$ be a prime number. Let $X$ be an $\Ff_p$-valued random
  variable, and let~$Y=X_1-X_2$ be a stepping of~$X$ and~$\widehat{Y}$ a
  peaking of~$X$.

  Let~$\eta>0$ be a real number. There exists $\beta>0$, depending
  only on~$\eta$, such that
  \begin{equation}\label{eq-estimate}
    \expect(|\varphi_X(X\widehat{Y})|^2)\ll
    \rho_X(0)+\rho_Y(0)^{\beta}+
    \frac{p^{-1+\eta}}{\rho_Y(0)}.
  \end{equation}
\end{proposition}

\begin{remark}\label{rm-comments}
  (1) To get a feeling for this inequality, note the obvious lower
  bounds
  $$
  \expect(|\varphi_X(X\widehat{Y})|^2)\geq \proba(X=0),\quad\quad
  \expect(|\varphi_X(X\widehat{Y})|^2)\geq \proba(\widehat{Y}=0).
  $$

  The term $\rho_X(0)$ on the right-hand side of~(\ref{eq-estimate})
  accounts for the first of these, and the third term accounts for (a
  quantity larger than) the second, since by~(\ref{eq-rho1}), we have
  $$
  \proba(\widehat{Y}=0)=\frac{1}{M_X}=\frac{p^{-1}}{\rho_Y(0)}.
  $$

  (2) Although the bound~(\ref{eq-estimate}) may look conventional
  enough, it is in its proof that additive combinatorics is crucial. In
  other words: if~(\ref{eq-estimate}) could be proved ``with classical
  means'', i.e. without invoking the sum-product phenomenon, or the
  Balog--Szemerédi--Gowers Theorem, or other results from additive
  combinatorics, then this would give a ``classical'' proof of
  Theorem~\ref{th-bgk}.
  
  (3) In ``concrete'' terms, without probabilistic notation, the
  quantity $\expect(|\varphi_X(X\widehat{Y})|^2)$ is the average
  $$
  \frac{1}{M_X^2}\sum_{x\in\Ff_p}\sum_{a\in\Ff_p}\rho_X(x)
  |\varphi_X(a)|^2|\varphi_X(ax)|^2.
  $$

  From an analytic number theory point of view, this can be interpreted
  as a kind of ``amplified'' average of the values of
  $|\varphi_X|^2$. To see why this can be useful, take the random
  variable~$X$ to be uniformly distributed over a subgroup~$H$
  of~$\Ff_p^{\times}$. Observe (as we will repeat later) that
  $\varphi_{X}(ah)=\varphi_X(a)$ for any~$h\in H$ and~$a\not=0$; it
  follows that $\rho_Y(0)=1/|H|$, and a simple computation shows that
  $X\widehat{Y}$ is distributed like~$Y$ and that
  $$
  M_X=\sum_{a\in\Ff_p}\Bigl|\frac{1}{p} \sum_{x\in
    H}e\Bigl(\frac{ax}{p}\Bigr)\Bigr|^2=\frac{p}{|H|}.
  $$

  Therefore, for any~$a\in\Ff_p^{\times}$, we have a lower
  bound   
  $$
  \expect(|\varphi_X(X\widehat{Y})|^2)\geq
  |\varphi_X(a)|^2\proba(X\widehat{Y}\in H)\geq |\varphi_X(a)|^2
  \times
  |H|\frac{|\varphi_X(a)|^2}{M_X}=|\varphi_X(a)|^4\frac{|H|^2}{p}. 
  $$

  This shows that even the trivial
  bound~$\expect(|\varphi_X(X\widehat{Y})|^2)\leq 1$ is sufficient to
  deduce that $|\varphi_X(a)|^4\leq p|H|^{-2}$, which is non-trivial as
  soon as~$H$ has size a bit larger than $\sqrt{p}$ -- the same range in
  which a ``direct'' use of Gauss sums leads to a non-trivial bound.

  Furthermore, if we apply Proposition~\ref{pr-alt-1} instead of the
  trivial bound, with $\eta=\gamma/2$, say, then we get some $\beta>0$
  such that
  $$
  \frac{|H|^2}{p}|\varphi_X(a)|^4\ll
  \frac{1}{|H|}+\frac{1}{|H|^{1+\beta}} +\frac{|H|p^{\eta}}{p}
  $$
  hence
  $$
  |\varphi_X(a)|^4 \ll p^{1-3\gamma}+p^{1-(2+\beta)\gamma}+p^{-\gamma/2},
  $$
  which proves Theorem~\ref{th-bgk} when $|H|=p^{\gamma}$
  with~$\gamma>\max(\tfrac{1}{3},\tfrac{1}{2+\beta})$, hence also for
  $\gamma$ slightly smaller than $1/2$. \emph{This is already a highly
    non-trivial fact}. A result of that type was first proved by
  Shparlinski~\cite{shparlinski} (for $|H|$ a bit larger than
  $p^{3/7}$), using estimates of Garcia and Voloch on the number of
  points on Fermat curves over finite fields, also combined with a
  fourth moment computation.
  %% (for $|H|$ a bit larger than $p^{3/7}$).
\end{remark}

\par\medskip\par

\textbf{Step 2.} We now describe for which random variables we will
apply Proposition~\ref{pr-exp-1}.  Let~$H\subset \Ff_p^{\times}$ be a
multiplicative subgroup.  We fix a random variable~$S$ which is
uniformly distributed on~$H$ (so that $\rho_S(x)=0$ unless $x\in H$, in
which case $\rho_S(x)=1/|H|$). We denote by $(S_k)_{k\geq 1}$ a sequence
of independent random variables, all independent from~$S$ and also
uniformly distributed on~$H$.

We will consider the random variables
$$
X_k=S_1-S_2+\cdots +S_{2k-1}-S_{2k}
$$
for $k\geq 1$. Probabilistically, these correspond to a simple random
walk on~$\Ff_p$ where the steps are taken alternately from~$H$ and
from~$-H$ (so the picture could be simplified a bit in the case where
$-1\in H$, since then each $S_i$ would be distributed in the same way as
$-S_i$, and we would have a ``standard'' random walk). Note that
$$
\varphi_{X_k}(a)=|\varphi_{S}(a)|^{2k},
$$
by independence; moreover, note that
$$
X_{2k}=(S_1-S_2+\cdots +S_{2k-1}-S_{2k})-(S_{2k+2}-S_{2k+1}+\cdots
+S_{4k}-S_{4k-1}),
$$
which shows that $X_{2k}$ is a stepping of~$X_k$.

For $\nu>0$, we define the set
$$
\Lambda_{\nu}=\{a\in\Ff_p\,\mid\, |\varphi_{S}(a)|>p^{-\nu}\}.
$$

Note that $0\in\Lambda_{\nu}$ in all cases, and that, since
$$
\varphi_S(a)=\frac{1}{|H|}\sum_{x\in H}e\Bigl(\frac{ax}{p}\Bigr),
$$
we can restate Theorem~\ref{th-bgk} as claiming the existence of
some~$\nu>0$ such that $\Lambda_{\nu}$ \emph{only} contains~$0$. This is
therefore our objective. The following simple lemma encapsulates the
specific property of the distribution of the random variable $S$.

\begin{lemma}\label{lm-uniform-properties}
  For any $x\in H$, the random variable~$xS$ is uniformly distributed
  on~$H$.

  In particular the following properties hold:
  \begin{enumth}
  \item For any $a\in\Ff_p$, we have $\varphi_S(ax)=\varphi_S(a)$, and
    hence also $\varphi_{X_k}(ax)=\varphi_{X_k}(a)$.
  \item The set $\Lambda_{\nu}\setminus\{0\}$ is either empty or is a
    union of $H$-cosets. In the second case, we have
    $|\Lambda_{\nu}|\geq |H|$.
  \end{enumth}
\end{lemma}

\begin{proof}
  The first statement simply reflects the fact that~$H$ is a
  multiplicative subgroup of~$\Ff_p^{\times}$. The equality
  $\varphi_{S}(ax)=\varphi_S(a)$ follows, and it means that
  $aH\subset \Lambda_{\nu}$ whenever $a\in\Lambda_{\nu}\setminus\{0\}$,
  which gives the last fact.
\end{proof}

The content of the second step is as follows:

\begin{proposition}\label{pr-alt-2}
  Let $\theta>0$ be a real number. If $p$ is a large enough prime
  number, depending only on~$\theta$, then there exist a positive real
  number $\nu<\demi\theta$, depending only on~$\theta$, and an
  integer~$k\geq 1$ such that
  \begin{equation}\label{eq-m-bound}
    p^{-1-\theta}|\Lambda_{\nu}|\leq\rho_{X_{2k}}(0)
    \leq p^{-1+\theta}|\Lambda_{\nu}|
  \end{equation}
  and
  \begin{equation}\label{eq-last-bound}
    \expect(|\varphi_{X_k}(X_k\widehat{X}_{2k})|^2)\geq p^{-10\theta}.
  \end{equation}
\end{proposition}

\par\medskip\par

\textbf{Step 3.} We now conclude the proof of
Theorem~\ref{th-bgk}. Recall that $|H|\geq p^{\gamma}$ by assumption; we
pick $\theta>0$ such that $10\theta<\gamma$. Applying
Proposition~\ref{pr-alt-2} and then Proposition~\ref{pr-alt-1}, for some
$\eta>0$ to be determined later, we find random variables $X=X_k$ and
$Y=X_{2k}$ satisfying the bounds~(\ref{eq-m-bound}) and such that
$$
p^{-10\theta}\leq \expect(|\varphi_X(X\widehat{Y})|^2)=
\expect(|\varphi_{X_k}(X_k\widehat{X}_{2k})|^2) \ll
\rho_X(0)+\rho_Y(0)^{\beta}+\frac{p^{-1+\eta}}{\rho_Y(0)},
$$
for some $\beta>0$.

The first term is easily handled: by induction on~$k$, we find that
$$
\proba(X_k=0)\leq \max_{x\in\Ff_p}\proba(S=x)=\frac{1}{|H|}
$$
for any $k\geq 1$, hence the assumption $|H|\geq p^{\gamma}$ gives
$$
p^{-10\theta}\ll p^{-\gamma}+\rho_Y(0)^{\beta}+
\frac{p^{-1+\eta}}{\rho_Y(0)}.
$$

Using~(\ref{eq-m-bound}) to estimate $\rho_Y(0)$ in terms of
$|\Lambda_{\nu}|$, this becomes
$$
p^{-10\theta}\ll
p^{-\gamma}+\Bigl(\frac{|\Lambda_{\nu}|}{p^{1-\theta}}\Bigr)^{\beta}
+\frac{p^{\eta+\theta}}{|\Lambda_{\nu}|}.
$$

We always have $|\Lambda_{\nu}|\leq p^{1+2\nu}|H|^{-1}$ by Chebychev's
inequality. Moreover, if we assume that~$\Lambda_{\nu}$ is not reduced
to~$0$, then this set contains at least $|H|\geq p^{\gamma}$
elements. Recalling that $2\nu<\eta$, we would then get the bounds
$$
p^{-10\theta}\ll p^{\beta(2\nu+\theta-\gamma})+p^{\eta+\theta-\gamma} 
\ll p^{\beta(\eta+\theta-\gamma)}+p^{\eta+\theta-\gamma},
$$
which is impossible for~$p$ large enough if $\eta$ is chosen small
enough in terms of $\gamma$.  Thus we must have $\Lambda_{\nu}=\{0\}$,
and (by definition) this means that
$$
\Bigl|\frac{1}{|H|}\sum_{x\in H}e\Bigl(\frac{ax}{p}\Bigr)\Bigr|\leq
p^{-\nu}
$$
for all~$a\in\Ff_p^{\times}$, provided $p$ is large enough.

\section{Completion of the proof}

We now prove Propositions~\ref{pr-alt-1} and~\ref{pr-alt-2}. The
sum-product theorem appears decisively in the proof of the first of
these, and more precisely in the following key proposition.

\begin{proposition}\label{pr-exp-1}
  Let $p$ be a prime number. Let $X$ be an $\Ff_p$-valued random
  variable, and let~$Y=X_1-X_2$ be a stepping of~$X$ as above. Let
  $\alpha\geq 1$ be a real number such that
  \begin{equation}\label{eq-bgk-cond2}
    \expect(\rho_Y(XY))\geq \frac{\rho_Y(0)}{\alpha}.
  \end{equation}

  Assuming that
  \begin{equation}\label{eq-bgk-cond1}
    \proba(X=0)\leq \frac{1}{4\alpha},\quad\quad \proba(Y=0)\leq
    \frac{1}{4\alpha},
  \end{equation}
  there exists a subset $A\subset \Ff_p^{\times}$ such that
  $$
  %% \frac{1}{\alpha^d \rho_Y(0)}\ll |A|\ll\frac{\alpha }{\rho_Y(0)}
  \frac{1}{2^{31}\alpha^{10} \rho_Y(0)}\leq |A|\leq
  \frac{8\alpha }{\rho_Y(0)}
  $$
  with the property that
  $$
  \max(|A+A|,|A\cdot A|)\leq 2^{878}\alpha^{294}|A|.
  $$
\end{proposition}

\begin{remark}
  As already indicated, the constants should really be interpreted as
  being of the form $c\alpha^d$ for some absolute constants $c>0$ and
  $d>0$.
\end{remark}

% \begin{remark}
%   Kurlberg~\cite[Prop.\,3.1]{kurlberg} shows that one can take here
%   $d=768$, and gives explicit values for all the implicit constants.
% \end{remark}

% (2) In the measure-theoretic language of the original paper, the
% assumption~(\ref{eq-bgk-cond2}) would be phrased as
% $$
% \sum_{x\in\Ff_p}\sum_{y\in\Ff_p}\rho_X(x)\rho_Y(y)\rho_Y(xy)\geq
% \alpha^{-1}.
% $$
% \end{remark}

% Intuitively, we could also rephrase the main
% assumption~(\ref{eq-bgk-cond2}) in the form
% $$
% \expect\Bigl(\proba(Y=XY')\Bigr)\geq \alpha^{-1}
% $$
% for some random variable~$Y'$ distributed like~$Y$. This can interpreted
% as saying that, on average, there is a rather large probability that
% $X=Y_1Y_2^{-1}$, with $Y_1$ and $Y_2$ distributed like~$Y$ (ignoring the
% possibility that~$Y_2=0$). If~$Y_1$ and~$Y_2$ were uniformly distributed
% on some subset~$A$ of $\Ff_p^{\times}$, this would amount to saying that
% the average of the representation function $r_{A\cdot A^{-1}}(X)$ is
% large, and we could then use the Balog--Szemerédi--Gowers Theorem to
% extract from~$A$ a subset with small product set. The proof of
% Proposition~\ref{pr-exp-1} will start by showing that finding such
% an~$A$ is indeed possible even without such a normalizing assumption
% on~$Y$, and then to show that the sumset $A+A$ is \emph{also} under
% control.

\begin{remark}
  % (1) Again, in concrete terms, with $\mu(x)=\proba(X=x)$ and
  % $\phi(y)=\proba(Y=y)$, this statement means that
  % $$
  % \sum_{x\in\Ff_p}\sum_{y\in\Ff_p}\mu(x)\phi(y)\phi(xy)= \frac{M_X}{p}
  % \sum_{x\in\Ff_p}\sum_{a\in\Ff_p}\varphi_X(ax)|\varphi_X(a)|^2\mu(x).
  % $$
  % This can be checked straightforwardly using elementary properties
  % of the Fourier transform.
  The use of the random variable $\widehat{Y}$ (which emphasizes
  values~$a\in\Ff_p$ where $|\varphi_X(a)|^2$ is ``large'') is
  reminiscent of the similar use of a non-uniform distribution in the
  proof of Theorem~\ref{th-schoen}.
  % (2) We obtain in particular the formula $M_X=p\rho_Y(0)$, so that
  % $$
  % \rho_{\widehat{Y}}(a)=\proba(\widehat{Y}=a)=
  % \frac{1}{p}\frac{\varphi_Y(a)}{\rho_Y(0)}
  % $$
  % for any~$a\in\Ff_p$.
\end{remark}

\begin{proof}%%[Proof of Proposition~\ref{pr-exp-1}]
  We will use frequently the fact that $\rho_Y(y)\leq \rho_Y(0)$ for
  all~$y\in \Ff_p$, which we already mentioned.
  % follows for instance from the first formula in
  % Lemma~\ref{lm-bgk-link}.

  We define
  $$
  A_1=\Bigl\{y\in\Ff_p\,\mid\, \rho_Y(y)\geq
  \frac{\rho_Y(0)}{8\alpha}\Bigr\}
  $$
  and $A_2=A_1\setminus\{0\}\subset \Ff_p^{\times}$ (note
  that~$0\in A_1$).  The main properties of~$A_2$ are given by the next
  lemma.

  \begin{lemma}
    We have
    \begin{equation}\label{eq-bgk-lb2}
      \frac{1}{4\alpha \rho_Y(0)}\leq |A_2|\leq \frac{8\alpha}{\rho_Y(0)},    
    \end{equation}
    and the representation function~$r_2$ for $A_2\cdot A_2^{-1}$
    satisfies
    \begin{equation}\label{eq-r2-lb}
      \expect(r_2(X))\geq \frac{|A_2|}{32\alpha^2}.
    \end{equation}
  \end{lemma}

  \begin{proof}
    First, simply by Chebychev's inequality, we have
    \begin{equation}\label{eq-a2-up}
      |A_2|\leq |A_1|\leq \frac{8\alpha}{\rho_Y(0)}.
    \end{equation}
    
    We now claim that the assumption~(\ref{eq-bgk-cond2}), namely
    $$
    \expect(\rho_Y(XY))\geq \frac{\rho_Y(0)}{\alpha},
    $$
    together with~(\ref{eq-bgk-cond1}), implies that
    \begin{equation}\label{eq-bgk-cond3}
      \expect(\rho_Y(XY)\charfun_{X\not=0,\, Y\in A_1\cap
        X^{-1}A_1})\geq \frac{\rho_Y(0)}{2\alpha}.
    \end{equation}
    
    This is a matter of showing that the contributions
    to~$\expect(\rho_Y(XY))$ from the complementary event, where $X=0$
    or $Y\notin A_1$, or $XY\notin A_1$, are small enough. And indeed,
    first of all the first part of~(\ref{eq-bgk-cond1}) gives the upper
    bound
    $$
    \expect\Bigl(\rho_Y(XY)\charfun_{X=0}\Bigr)=\rho_Y(0)\proba(X=0) \leq
    \frac{\rho_Y(0)}{4\alpha},
    $$
    while
    $$
    \expect(\rho_Y(XY)\charfun_{X\not=0,\ XY\notin A_1}) \leq
    \frac{1}{8\alpha}\expect(\rho_Y(XY))\leq \frac{\rho_Y(0)}{8\alpha}.
    $$
    
    To bound the last contribution with $X\not=0$ and $Y\notin A_1$, we
    write
    $$
    \expect(\rho_Y(XY)\charfun_{X\not=0,\ Y\notin A_1}) =\sum_{y\notin
      A_1} \expect(\rho_Y(XY)\charfun_{X\not=0,\, Y=y})= \sum_{y\notin
      A_1} \expect(\rho_Y(yX)\charfun_{X\not=0,\, Y=y}).
    $$
    
    Using the independance of~$X$ and~$Y$, we deduce that
    \begin{multline*}
      \expect(\rho_Y(XY)\charfun_{X\not=0,\ Y\notin A_1}) =\sum_{y\in
        \Ff_p\setminus A_1}\proba(Y=y)\expect(\rho_Y(yX)\charfun_{X\not=0})
      \\
      \leq \frac{\rho_Y(0)}{8\alpha}\expect\Bigl(\sum_{y\notin
        A_1}\rho_Y(yX)\charfun_{X\not=0} \Bigr) \leq \frac{\rho_Y(0)}{8\alpha}
      \expect\Bigl(\sum_{y\in\Ff_p}\rho_Y(yX)\charfun_{X\not=0}\Bigr) \leq
      \frac{\rho_Y(0)}{8\alpha},
    \end{multline*}
    using in the last step the fact that, for any given $x\not=0$, we
    have
    $$
    \sum_{y\in\Ff_p}\rho_Y(yx)=\proba(Y\not=0)\leq 1.
    $$
    
    We next deduce from~(\ref{eq-bgk-cond3}) a lower-bound for~$|A_1|$
    complementing the upper-bound~(\ref{eq-a2-up}), namely
    \begin{equation}\label{eq-bgk-lb1}
      \frac{1}{2\alpha \rho_Y(0)}\leq |A_1|\leq \frac{8\alpha}{\rho_Y(0)},
    \end{equation}
    which in turn implies that $|A_1|\geq 2$ (by~(\ref{eq-bgk-cond1})
    since $\rho_Y(0)=\proba(Y=0)$), and therefore also
    $|A_2|=|A_1|-1\geq \demi |A_1|$, hence
    $$
    \frac{1}{4\alpha \rho_Y(0)}\leq |A_2|\leq
    \frac{8\alpha}{\rho_Y(0)},
    $$
      
    Indeed, we obtain~(\ref{eq-bgk-lb1}) by noting that,
    by~(\ref{eq-bgk-cond3}), we have
    $$
    \frac{\rho_Y(0)}{2\alpha}\leq \expect(\rho_Y(XY)\charfun_{X\not=0,\
      Y\in A_1})\leq \rho_Y(0)\proba(Y\in A_1)\leq \rho_Y(0)^2|A_1|.
    $$
    
    The next step is to relate the bound~(\ref{eq-bgk-cond3}) to the
    representation function~$r_2$ for~$A_2\cdot A_2^{-1}$. For this, we
    start with the formula
    $$
    \expect(r_2(X))=\sum_{y,z\in A_2}\proba(X=y^{-1}z)=
    \sum_{y\in A_2}\expect\Bigl(\sum_{z\in A_2}\proba(yX=z)\Bigr)=
    \sum_{y\in A_2}\proba(yX\in A_2).
    $$
    
    On the other hand, by independance of~$X$ and~$Y$, we have
    \begin{multline*}
      \expect(\rho_Y(XY)\charfun_{X\not=0,\, Y\in A_1\cap X^{-1}A_1})
      =\sum_{y\in A_1}\rho_Y(y)\expect(\rho_Y(yX)\charfun_{X\not=0,\,
        yX\in A_1})
      \\
      \leq \rho_Y(0)^2 \expect\Bigl(\sum_{y\in A_1}\charfun_{X\not=0,\
        yX\in A_1}\Bigr)=\rho_Y(0)^2\sum_{y\in A_1}\proba(X\not=0\text{
        and } yX\in A_1).
    \end{multline*}
    
    Isolating the contribution of~$y=0\in A_1$, we then have
    $$
    \sum_{y\in A_1}\proba(X\not=0\text{ and }
    yX\in A_1)=\proba(X\not=0)+\expect(r_2(X))\\
    \leq 1+\expect(r_2(X)),
    $$
    and thus~(\ref{eq-bgk-cond3}) implies that
    $$
    \frac{\rho_Y(0)}{2\alpha}\leq 
    \rho_Y(0)^2\expect(r_2(X))+\rho_Y(0)^2.
    $$
    
    The assumption $\proba(Y=0)=\rho_Y(0)\leq (4\alpha)^{-1}$
    (see~(\ref{eq-bgk-cond1})) now leads to the lower-bound
    $$
    \expect(r_2(X))\geq \frac{1}{4\alpha \rho_Y(0)}\geq
    \frac{|A_2|}{32\alpha^2},
    $$
    concluding the proof.
  \end{proof}
  
  Using~(\ref{eq-r2-lb}), we can apply Lemma~\ref{lm-random-energy} to
  the random variable~$X$ on~$\Ff_p^{\times}$, with $\beta=32\alpha^2$;
  we obtain
  $$
  e(A_2)\geq \frac{1}{2^{22}\alpha^8\rho_Y(0)|A_2|}\geq
  \frac{1}{2^{25}\alpha^9},
  $$
  and therefore, by the Balog--Szemerédi--Gowers Theorem
  (Theorem~\ref{th-schoen}, applied to~$A_2\subset \Ff_p^{\times}$),
  there exists a subset $A_3\subset A_2$ with
  $$
  |A_2|\leq 4(2^{25}\alpha^9)|A_3|=2^{27}\alpha^9|A_3|,\quad\quad
  |A_3\cdot A_3|\leq 2^{14}(2^{25}\alpha^9)^6 |A_3|=2^{164}\alpha^{54}|A_3|.
  $$

  But we can also control the additive properties of~$A_3$. Precisely,
  we can apply Lemma~\ref{lm-stepping} to the group~$\Ff_p$, the random
  variables $X$ and~$Y$, and the set~$B=A_3$, with parameters
  $(\alpha,\beta)=(8\alpha,2^{29}\alpha^{10})$, since~$A_3\subset A_1$
  and
  $$
  |A_3|\geq \frac{|A_2|}{2^{27}\alpha^9}\geq
  \frac{1}{2^{29}\alpha^{10}\rho_Y(0)}
  $$
  thanks to~(\ref{eq-bgk-lb2}). The conclusion is that
  $$
  e(A_3)\geq
  \frac{1}{4(8\alpha)^9(2^{29}\alpha^{10})^4}=\frac{1}{2^{144}\alpha^{49}}.
  $$

  Applying Theorem~\ref{th-schoen} to~$A_3 \subset \Ff_p$, we find a
  subset $A_4\subset A_3$ with $|A_3|\leq 4\alpha|A_4|$ and
  $$
  |A_4+A_4|\leq 2^{14}(2^{144}\alpha^{49})^6
  |A_4|=2^{878}\alpha^{294}|A_4|.
  $$

  Since, in addition, we have
  $$
  |A_4\cdot A_4|\leq |A_3\cdot A_3|\leq 2^{164}\alpha^{54}|A_3|\leq
  2^{166}\alpha^{55}|A_4|,
  $$
  and
  $$
  \frac{1}{2^{31}\alpha^{10}\rho_Y(0)}\leq \frac{|A_3|}{4\alpha}\leq
  |A_4|\leq |A_3|\leq \frac{8\alpha}{\rho_Y(0)},
  $$
  we finally have proved Proposition~\ref{pr-exp-1} with the set~$A$
  equal to~$A_4$.
\end{proof}

In order to prove Proposition~\ref{pr-alt-1}, we combine this with a
consequence of Lemma~\ref{lm-bgk-link}, using Fourier analysis to obtain
a ``diophantine'' interpretation of
$\expect(|\varphi_X(X\widehat{Y})|^2)$.

\begin{lemma}\label{lm-bgk-link2}
  We have
  $$ \expect(\rho_Y(XY))=\rho_Y(0) \expect(|\varphi_X(X\widehat{Y})|^2).
  $$
\end{lemma}

\begin{proof}
  Using the formula $\rho_Y(0)=M_X/p$ and Lemma~\ref{lm-bgk-link}, we
  have
  $$
  \expect(\rho_Y(XY))=\rho_Y(0)\expect(\varphi_{\widehat{Y}}(XY)),
  $$
  and it only remains to appeal to the symmetry formula
  $$
  \expect(\varphi_{\widehat{Y}}(XY))=
  \expect(|\varphi_X(X\widehat{Y})|^2)
  $$
  to conclude. This last identity can be seen as a (very simple)
  instance of Fubini's formula:
  \begin{align*}
    \expect(\varphi_{\widehat{Y}}(XY))=
    \expect\Bigl(\expect\Bigl(e\Bigl(\frac{XY\widehat{Y}}{p}\Bigr)
    \Bigr)\Bigr) &=\expect\Bigl(
    \expect\Bigl(e\Bigl(\frac{X(X_1-X_2)\widehat{Y}}{p}\Bigr)\Bigr)\Bigr)
    \\
    &=\expect\Bigl(\Bigl|\expect\Bigl(e\Bigl(\frac{XX_1\widehat{Y}}{p}
    \Bigr)\Bigr)\Bigr|^2\Bigr)=\expect(|\varphi_X(X_1\widehat{Y}))|^2),
  \end{align*}
  leading to the conclusion since $X$ and~$X_1$ are identically
  distributed.
\end{proof}

\begin{proof}[Proof of Proposition~\ref{pr-alt-1}]
  We define $\alpha\geq 1$ by
  $\expect(|\varphi_X(X\widehat{Y})|^2)=\alpha^{-1}$. 
  By Lemma~\ref{lm-bgk-link2}, we have then
  $$
  \expect(\rho_Y(XY))=\frac{\rho_Y(0)}{\alpha}.
  $$

  If the conditions~(\ref{eq-bgk-cond1}) are not valid, then by
  construction this implies that the bound
  $$
  \expect(|\varphi_X(X\widehat{Y})|^2)=\alpha^{-1}\leq
  4(\rho_X(0)+\rho_Y(0))
  $$
  holds.  On the other hand, if these conditions are satisfied, then we
  can apply Proposition~\ref{pr-exp-1} to deduce the existence of
  $A\subset \Ff_p^{\times}$ with
  $$
  \max(A+A,A\cdot A)\ll \alpha^d|A|
  $$
  and
  $$
  \frac{1}{\alpha^d\rho_Y(0)}\ll |A|\ll \frac{\alpha}{\rho_Y(0)},
  $$
  where $d$ and the implied constants are absolute (and explicit).

  Let~$\eta>0$. We distinguish two further cases:

  (1) If $|A|\leq p^{1-\eta}$, then denoting by $\delta>0$ the exponent
  in Theorem~\ref{th-sum-product} for $\gamma=\eta$, we have
  $\alpha^d\gg |A|^{\delta}$. It follows that
  $\alpha^d\gg \alpha^{-d\delta}\rho_Y(0)^{-\delta}$, and hence
  $$
  \expect(|\varphi_X(X\widehat{Y})|^2)=\alpha^{-1}\ll
  \rho_Y(0)^{\delta/(d+d\delta)}.
  $$

  (2) If $|A|>p^{1-\eta}$, then
  $$
  \expect(|\varphi_X(X\widehat{Y})|^2)=\alpha^{-1}\ll
  \frac{1}{|A|\rho_Y(0)}\ll \frac{p^{-1+\eta}}{\rho_Y(0)}.
  $$

  All three of the bounds thus obtained imply that the
  estimate~(\ref{eq-estimate}) holds (with
  $\beta=\min(1,\delta/(d+d\delta))$), concluding the proof.
\end{proof}

We now come to the proof of Proposition~\ref{pr-alt-2}. Only in the last
step will the specific properties of the distribution of~$S$ be
important.
% This splits into
% a fairly general part where the distribution of the steps of the random
% walk are quite arbitrary, and which provides the
% estimate~(\ref{eq-m-bound}) for a suitable value of~$k$, and a last step
% where~(\ref{eq-last-bound}) is obtained when this distribution is
% uniform on a multiplicative subgroup.

\begin{proof}[Proof of Proposition~\ref{pr-alt-2}]
We recall the definition
$$
X_k=\sum_{i=1}^k (S_{2i-1}-S_{2k}),\quad\quad k\geq 1,
$$
of the random walk and the formula $\varphi_{X_k}=|\varphi_S|^{2k}$.

We observe first that for any integer~$k\geq 1$ and $\nu>0$, provided
the condition $4k\nu\leq \theta$ is satisfied, the estimate
\begin{equation}\label{eq-mbound3}
  \rho_{X_{2k}}(0)=\frac{M_{X_k}}{p}=
  \frac{1}{p}\sum_{a\in\Ff_p}|\varphi_{S}(a)|^{4k}\geq
  |\Lambda_{\nu}|p^{-1-\theta}
\end{equation}
holds by~(\ref{eq-rho1}) and the definition of~$\Lambda_{\nu}$.

We now claim that if $p$ is large enough, depending only on~$\theta$,
then we can find some integer~$k\geq 1$ and $\nu<\demi\theta$,
independent of~$p$, such that $4k\nu\leq \theta$ and
\begin{equation}\label{eq-m-bound2}
  p^{-\theta} \leq  \frac{|\Lambda_{\nu}|}{M_{X_k}},
\end{equation}
which, together with~(\ref{eq-mbound3}) and the formula
$\rho_{X_{2k}}(0)=M_{X_k}/p$, ensures that~(\ref{eq-m-bound}) holds for
these choices of~$k$ and~$\nu$.

To prove the claim, we first note that there is a general upper bound
$$
M_{X_k}\leq |\Lambda_{1/k}|+p\cdot (p^{-4k})^k=|\Lambda_{1/k}|
+p^{-3}\leq |\Lambda_{1/k}|(1+p^{-3}),
$$
valid for any integer~$k\geq 1$.  Now, given~$k\geq 1$, we denote
$k_+=\lceil \frac{\theta}{k^2}\rceil$.  If the inequality
$M_{X_{k}}>p^{\theta}|\Lambda_{1/k_+}|$ holds, then it follows that
$$
|\Lambda_{1/k_+}|\leq |\Lambda_{1/k}|p^{-\theta}(1+p^{-3}).
$$

Iterating this observation $m$ times, starting from~$k=4$, we see that
\emph{either} we find~$k\geq 1$ such that~(\ref{eq-m-bound2}) holds for
$\nu=1/k_+$, or we have
$$
|\Lambda_{1/k}|\leq p^{1-m\theta}(1+p^{-3})^m
$$
for $m\geq 1$ and some~$k$ depending on~$m$. But for suitable $m$, we
obtain $|\Lambda_{1/k}|<1$, which is a contradiction since
$0\in\Lambda_{\nu}$ for all~$\nu$.

Our next goal is the inequality
\begin{equation}\label{eq-expansion}
  \expect(|\varphi_{X_k}(a X_k)|^2)\geq \varphi_{S}(a)^{4k}
\end{equation}
for all~$k\geq 1$ and $a\in\Ff_p$, and this will depend on the specific
choice of the random walk. Indeed, we first have
$$
\expect(|\varphi_{X_k}(a X_k)|^2)= \expect(\varphi_{X_k}(a
X_{2k}))=\expect(|\varphi_S (a X_{2k})|^{2k}) \geq \expect(\varphi_S(a
X_{2k}))^{2k},
$$
by Jensen's inequality. However, by a discrete Fubini, we have
$$
\expect(\varphi_S(a X_{2k}))=\expect(|\varphi_{X_{k}}(a
S)|^2)
$$
and $\expect(|\varphi_{X_k}(a S)|^2)=\varphi_{X_k}(a)^2$ since
$\varphi_{X_k}(a S)=\varphi_{X_k}(a)$ (the crucial fact from
Lemma~\ref{lm-uniform-properties}), which gives~(\ref{eq-expansion}).

We can then finally
deduce~(\ref{eq-last-bound}). From~(\ref{eq-m-bound2}) and the condition
$4k\nu\leq \theta$, we deduce the lower bound
$$
\proba(\widehat{X}_{2k}\in\Lambda_{\nu})\geq
p^{-\theta}\frac{|\Lambda_{\nu}|}{M_{X_k}}\geq p^{-2\theta},
$$
and then from~(\ref{eq-expansion}), we get
$$
\expect(|\varphi_{X_k}(X_k\widehat{X}_{2k})|^2) \geq
\expect(\varphi_{X_k}(\widehat{X}_{2k})^{4k}) \geq
p^{-4k^2\nu}\proba(\widehat{X}_{2k}\in\Lambda_{\nu}) \geq p^{-4k^2\nu
  -2\theta}\geq p^{-10\theta}.
$$
\end{proof}

\section{Remarks}

We conclude with a few brief remarks.
\par
(1) One interpretation of Theorem~\ref{th-bgk} is that it is one more
avatar of the fact that the additive and multiplicative structures of a
finite field (or of the integers) are fairly ``independent'': it
concerns the \emph{additive} Fourier transform of a
\emph{multiplicative} subgroup.  In this sense, it is of a flavor
comparable with the sum-product theorem.

One may however then wonder about exchanging the role of addition and
multiplication. And whereas the sum-product theorem is fully symmetric,
the ``dual'' of Theorem~\ref{th-bgk} would become the problem of
estimating sums of \emph{multiplicative} (Dirichlet) characters
modulo~$p$ over \emph{very short} intervals in~$\Ff_p$ -- a problem
which is intimately related with the Generalized Riemann Hypothesis and
properties of Dirichlet $L$-functions. (We see short intervals as
analogues of small multiplicative subgroups in view of their additive
properties, which makes them behave quite similarly to non-existent
small additive subgroups; this is reasonable especially because
Theorem~\ref{th-bgk} does extend to geometric progressions in addition
to multiplicative subgroups.)

Could the proof of Theorem~\ref{th-bgk} give insight about such
character sums?  This doesn't seem to be likely, because there is no
analogue of Lemma~\ref{lm-uniform-properties} (e.g., the existence of
\emph{one} large character sum for a non-trivial character does not, a
priori, lead to the existence of any other). Ultimately, this reflects
the fact that addition and multiplication \emph{are not} symmetric in
the definition of a field: multiplication is distributive with respect
to addition, and not the opposite, so that multiplication by non-zero
elements give automorphisms of the \emph{additive} group of a field,
leading to symmetry properties of the \emph{additive} Fourier transform
of multiplicative subgroups.
\par
(2) One can also ask if there are echoes in this proof of more classical
ideas in the study of exponential sums (such as those of Weyl, van der
Corput and Vinogradov, see e.g.~\cite[Ch.\,8]{ant}).

We see at least two clear links of this type:
\begin{itemize}
\item The use of $|\varphi_S|^2$ and higher powers is very much in the
  spirit of ``creating new points of summation'' or Weyl differencing.
\item The link in Lemma~\ref{lm-bgk-link2}, based on harmonic analysis,
  between averages of the Fourier transform and averages of the
  ``density'' $\rho_Y$ is an example of reduction of averages of
  exponential sums to point counting.
\end{itemize}

One related remark is that if we consider, instead of the crucial
expression $\expect(|\varphi_X(X\widehat{Y})|^2)$ in
Proposition~\ref{pr-alt-1}, the simpler
$\expect(|\varphi_X(\widehat{Y})|^2)$, then we get (up to normalization)
simply the fourth moment of $\varphi_X(a)$, instead of a kind of average
``twisted'' fourth moment.

% $$
% \expect(|\varphi_X(XU)|^2)=\frac{1}{p}\rho_X(0)+
% \frac{1}{p}(1-\rho_X(0))\rho_Y(0). 
% $$

\par
(3) Another parallel is with the work of Bourgain and
Gamburd~\cite{bourgain-gamburd} on expansion properties of Cayley graphs
of $\SL_2(\Ff_p)$, which is almost contemporary with
Theorem~\ref{th-bgk}. For instance, the crucial ``$L^2$-flattening
lemma'' of Bourgain and Gamburd~\cite[Prop.\,2]{bourgain-gamburd} can be
interpreted as a quantitative statement of decay of $\proba(Y=0)$ for a
stepping~$Y$ of certain random variables~$X$ on $\SL_2(\Ff_p)$.
Lemma~\ref{lm-uniform-properties} also has a similar flavor to the use
of the ``pseudo-randomness'' of $\SL_2(\Ff_p)$ (i.e., the absence of
non-trivial irreducible representations of small dimension)
in~\cite[Prop.\,1]{bourgain-gamburd}.


\begin{thebibliography}{CCCC}

% \bibitem{balog-szemeredi} A. Balog and E. Szemerédi: \textit{A
%     statistical theorem of set addition}, Combinatorica 14 (1994),
%   263--268.

\bibitem{bkt} J. Bourgain, N.H. Katz and T. Tao: \textit{A sum-product
    estimate in finite fields, and applications}, GAFA 14 (2004),
  27--57.
  
% \bibitem{bourgain} J. Bourgain: \textit{Exponential sum estimates over
%     subgroups of $\mathbb{Z}_q^*$, $q$ arbitrary}, Journal d'Analyse
%   Math\'{e}matique Vol. 97 (2005), 317--355.

\bibitem{bourgain-talk} J. Bourgain: \textit{Exponential sums,
    equidistribution and pseudo-randomness}, talk at I.A.S, December 3,
  2008; \url{https://www.youtube.com/watch?v=s1EhZQ5kSNw}.
  
\bibitem{bourgain-mordell} J. Bourgain: \textit{Mordell's exponential
    sums estimate revisited}, Journal A.M.S. 18 (2005), 477--499.

\bibitem{bgk} J. Bourgain, A.A. Glibichuk and S. Konyagin:
  \textit{Estimates for the number of sums and products and for
    exponential sums in fields of prime order}, J. London Math. Soc. 73
  (2006), 380--398.

\bibitem{bourgain-gamburd} J. Bourgain and A. Gamburd: \textit{Uniform
    expansion bounds for Cayley graphs of $\SL_2(\Ff_p)$}, Ann. of
  Math. 167 (2008), 625--642.

\bibitem{breuillard} E. Breuillard: \textit{A brief introduction to
    approximate groups}, in ``Thin groups and super-strong
  approximation'', edited by E. Breuillard and H. Oh, MSRI Publications
  Vol. 61, Cambridge Univ. Press, 2014.
  
% \bibitem{erdos-szemeredi} P. Erd\H os and E. Szemerédi: \textit{On sums
%     and products of integers}, in ``Studies in pure mathematics to the
%   memory of Paul Turán'', edited by P. Erd\H os, L. Alpár, G. Halász and
%   A. Sárközy, Studies in Pure Mathematics, Birkhäuser (1983).

\bibitem{ant} H. Iwaniec and E. Kowalski: \textit{Analytic Number
    Theory}, Colloquium Publ. 53, A.M.S, 2004.

\bibitem{add-comb} E. Kowalski: \textit{Introduction to additive
    combinatorics}, ETH lecture notes (2023);
  \url{https://www.math.ethz.ch/~kowalski/additive-combinatorics.pdf}
  
\bibitem{kurlberg} P. Kurlberg: \textit{Bounds on exponential sums over
    small multiplicative subgroups}, in ``Additive combinatorics'', CRM
  Proc. Lecture Notes, 43, A.M.S, 2007.

\bibitem{schoen} T. Schoen: \textit{New bounds in
    Balog--Szemerédi--Gowers}, Combinatorica 35 (2015), 695--701.

\bibitem{shkredov} I. Shkredov: \textit{Some remarks on the asymmetric
    sum-product phenomenon}, Moscow J. Comb. Number Th. 8 (2019),
  15--41.
  
\bibitem{shparlinski} I. Shparlinski: \textit{Estimates for Gauss sums},
  Mat. Zametki 50 (1991), 122--130.

\end{thebibliography}
\end{document}